\documentclass{amsart}

\usepackage{amsmath}
\usepackage{amssymb}
\usepackage{commath}
\usepackage{mathtools}
\usepackage{multirow}
\usepackage{xypic}

\usepackage{hyperref}

\allowdisplaybreaks

\theoremstyle{theorem}
\newtheorem{theorem}{Theorem}[section]
\newtheorem{corollary}[theorem]{Corollary}
\newtheorem{proposition}[theorem]{Proposition}
\newtheorem{lemma}[theorem]{Lemma}

\theoremstyle{definition}
\newtheorem{definition}[theorem]{Definition}
\newtheorem{remark}[theorem]{Remark}

\numberwithin{theorem}{section}
\numberwithin{equation}{section}

\newcommand{\C}{\mathbb{C}}
\newcommand{\R}{\mathbb{R}}
\newcommand{\T}{\mathbb{T}}
\newcommand{\B}{\mathbb{B}}

\newcommand{\Z}{\mathbb{Z}}
\newcommand{\HH}{\mathbb{H}}

\DeclareMathOperator{\im}{Im}
\DeclareMathOperator{\re}{Re}

\DeclareMathOperator{\arccot}{arccot}

\newcommand{\OO}{\mathrm{O}}

\newcommand {\Aa}    {\mathcal{A}}

\newcommand {\Fa}    {\mathcal{F}}

\newcommand {\Ma}    {\mathcal{M}}

\newcommand {\Ta}    {\mathcal{T}}

\title[Toeplitz and group-moment coordinates]{Toeplitz operators and group-moment coordinates for quasi-elliptic and quasi-hyperbolic symbols}

\author{Ra\'ul Quiroga-Barranco}
\address{Centro de Investigaci\'on en Matem\'aticas, Guanajuato, M\'exico}
\email{quiroga@cimat.mx}

\author{Armando S\'anchez-Nungaray}
\address{Facultad de Matem\'aticas, Universidad Veracruzana, Veracruz, M\'exico}
\email{sancheznungaray@gmail.com}



\begin{document}

\maketitle

\begin{abstract}
	For $\mathbb{B}^n$ the $n$-dimensional unit ball and $D_n$ its Siegel unbounded realization, we consider Toeplitz operators acting on weighted Bergman spaces with symbols invariant under the actions of the maximal Abelian subgroups of biholomorphisms $\mathbb{T}^n$ (quasi-elliptic) and $\mathbb{T}^n \times \mathbb{R}_+$ (quasi-hyperbolic). Using geometric symplectic tools (Hamiltonian actions and moment maps) we obtain simple diagonalizing spectral integral formulas for such kinds of operators. Some consequences show how powerful the use of our differential geometric methods are.
\end{abstract}


\section{Introduction}
The study of Bergman spaces on the unit ball $\B^n$ and Toeplitz operators acting on them is by now very well known to be tightly connected to ($n$-dimensional) M\"obius transformations. Of course, this includes the unbounded realization given by the $n$-dimensional Siegel domain $D_n$. 

Symplectic geometry has also shown its importance in this branch of operator theory. In particular, we have used Hamiltonian actions and moment maps in \cite{QSJFA} to study Toeplitz operators with symbols that are invariant under maximal Abelian subgroups of M\"obius transformations. The name of the latter family of subgroups is collectively abbreviated as MASG. In such previous work, we used symplectic geometric tools to obtain diagonalizing spectral integral formulas for the quasi-elliptic, quasi-parabolic, nilpotent and quasi-nilpotent cases. We refer to \cite{QSJFA,QVUnitBall1} for the notation. These sort of formulas already appeared in \cite{QVUnitBall1}. However, it was proved in \cite{QSJFA} that the use of symplectic geometry and moment maps highly simplifies the expressions of the spectral integral formulas as well as the computations needed to obtain them. This naturally yields a better understanding of Toeplitz operators.

On the other hand, there are some computational difficulties observed in the quasi-hyperbolic case (see \cite{QVUnitBall1}), which already appeared in the $1$-dimensional setup of the upper half-plane (see \cite{GQVJFA}). This has led to very complicated diagonalizing formulas whose proofs are very lengthy and complicated as well.

Similar computational difficulties appeared also in our use of symplectic geometric tools in \cite{QSJFA}, thus obstructing us from obtaining a diagonalizing spectral integral formula for the quasi-hyperbolic case in that work. We refer to \cite{QSJFA,QVUnitBall1} and Section~\ref{sec:analysisgeometry} below for the definition of the corresponding MASG.

The present work shows that the complexities of the quasi-hyperbolic case can indeed be overcome through the use of symplectic geometric tools. The main result, Theorem~\ref{thm:quasi-hyperbolic-spectral}, provides a simultaneous diagonalizing spectral integral formula for Toeplitz operators with symbols invariant under the quasi-hyperbolic action of $\T^{n-1} \times \R_+$ on $D_n$. Our formula holds on any dimension $n \geq 1$, but it remarkably applies in a non-trivial way in the case $n = 1$ in Corollary~\ref{cor:hyperbolic-n=1}. The latter provides a very simple formula, both the expression and proof, in comparison with those from previous works (see \cite{GQVJFA}). The simplicity of Theorem~\ref{thm:quasi-hyperbolic-spectral} allows us to obtain another interesting conclusion: the $C^*$-algebra generated by Toeplitz operators with quasi-hyperbolic symbols in dimension $n$, contains a $C^*$-subalgebra isomorphic to the one generated by Toeplitz operators with quasi-elliptic symbols in dimension $n-1$. Such is the content of Corollary~\ref{cor:quasi-hyperbolic-H_0}. We hope this sort of results might pave the way to a description of the $C^*$-algebras generated by Toeplitz operators with symbols invariant under a MASG for dimensions $n \geq 2$.

We also provide a diagonalizing spectral integral formula for the quasi-elliptic case in Theorem~\ref{thm:quasi-elliptic-spectral}. This allows us to better understand the connection between the quasi-elliptic case and the toral part of the quasi-hyperbolic case, thus leading to the proof of Corollary~\ref{cor:quasi-hyperbolic-H_0}.

We observe that the main breakthrough that allows us to obtain our (relatively) simple spectral integral formulas in the quasi-hyperbolic case is the use of a very useful analytic tool: the measure associated with the Romanovski-Routh polynomials, that is defined below in \eqref{eq:Romanovski-Routh}. We refer to \cite{RomanMartinez,SpecialFunctions,RomanPeetre,RomanWeber} for further details on the importance of such polynomials. Here, we will only mention that the weight associated to the measure \eqref{eq:Romanovski-Routh} is used with the Rodrigues formula to define them. We also note that Romanovski-Routh polynomials have been used before to study Toeplitz operators in \cite{SMRPolyBergmanRomanovski} within the setup of polyanalytic spaces.

An important contribution of this work lies behind the proof of our spectral integral formulas for Toeplitz operators with suitably invariant symbols. More precisely, such formulas are obtained by constructing Bargmann-type transforms based on the geometry provided by Hamiltonian actions and moment maps. This is achieved in Theorems~\ref{thm:Segal-Bargmann-QE} and \ref{thm:Segal-Bargmann-QH} for the quasi-elliptic case and the quasi-hyperbolic case, respectively. In the spirit of the classical Bargmann transform, we prove in those cases that the Bergman space is unitarily equivalent to full $L^2$ spaces in a way that strongly relates to the corresponding actions involved. And this is accomplished using almost exclusively symplectic geometric tools. Hence, this work shows the importance of such geometric tools as a background for Toeplitz operators on Bergman spaces.

As for the distribution of our results, Section~\ref{sec:analysisgeometry} states the basic facts needed from both functional analysis and differential geometry. Section~\ref{sec:actionmomentcoord} introduces coordinates on both $\B^n$ and $D_n$ that can be naturally associated to the MASG's given by the quasi-elliptic and quasi-hyperbolic cases, respectively. Such coordinates are obtained from the moment maps of the corresponding Hamiltonian actions. These constructions allow us to obtain in Section~\ref{sec:Bergman-groupmoment-coord} Bargmann-type transforms for the quasi-elliptic and quasi-hyperbolic cases. As discussed above, Section~\ref{sec:Toeplitz-momentmapsymbols} contains the main conclusions for Toeplitz operators and the $C^*$-algebras that they generate for the MASG's of interest in this work.

\section{Analysis and geometry on the unit ball and the Siegel domain} 
\label{sec:analysisgeometry}
We will denote by $\B^n$ the unit ball in $\C^n$. Recall that $\B^n$ admits an unbounded realization given by
\[
    D_n = \{ z = (z',z_n) \in \C^n \mid \im(z_n) - |z'|^2 > 0  \},
\]
called the $n$-dimensional Siegel domain. Note that here, and in the rest of this work, for every $z \in \C^n$ we write $z = (z',z_n)$, where $z' \in \C^{n-1}$ and $z_n \in \C$. 

As usual, we will denote by $\dif v(z)$ the Lebesgue measure on $\C^n$. For every $\lambda > -1$, we consider the measure 
\[
	\dif v_\lambda(z) = c_\lambda (1 - |z|^2)^\lambda \dif v(z),
\]
where the constant
\[
	c_\lambda = \frac{\Gamma(n+\lambda+1)}{\pi^n \Gamma(\lambda+1)}
		= \frac{(\lambda + 1)_n}{\pi^n}
\]
is chosen so that $v_\lambda(\B^n) = 1$. We recall that the expression $(x)_m$ denotes a Pochhammer symbol. Correspondingly, for the domain $D_n$ we define the measure
\[
	\dif \widehat{v}_\lambda(z) = 
		\frac{c_\lambda}{4}(\im(z_n) - |z'|^2)^\lambda \dif v(z).
\]
The weighted Bergman spaces with weight $\lambda > -1$ on $\B^n$ and $D_n$ are defined by
\begin{align*}
	\Aa_\lambda^2(\B^n) &= \{ f \in L^2(\B^n,v_\lambda) \mid 
				f \text{ is holomorphic } \}, \\
	\Aa_\lambda^2(D_n) &= \{ f \in L^2(D_n,\widehat{v}_\lambda) \mid 
				f \text{ is holomorphic } \},
\end{align*}
respectively. It is well known that both spaces are unitarily equivalent through the Cayley transform (see \cite{QVUnitBall1} for details). Furthermore, they are both closed, in $L^2(\B^n, v_\lambda)$ and $L^2(D_n,\widehat{v}_\lambda)$, respectively, as well as reproducing kernel Hilbert spaces. Their orthogonal projections are given by the expressions
\begin{align*}
	(B_{\B^n,\lambda})f(z) &= 
		\int_{\B^n} \frac{f(w) \dif 
			v_\lambda(z)}{(1-|z|^2)^{n+\lambda+1}}, \\
	(B_{D_n,\lambda})f(z) &=
		\int_{D_n} 
			\frac{f(w) \dif \widehat{v}_\lambda(z)}%
				{\Big(\frac{z_n - \overline{w}_n}{2i} 
					- z'\cdot \overline{w}'\Big)^{n+\lambda+1}},
\end{align*}
respectively. 

In the rest of this work, we will denote by $D$ either of the domains $\B^n$ or $D_n$ when we refer to properties shared by both cases. For every $a \in L^\infty(D)$ and $\lambda > -1$ the Toeplitz operator $T_a = T^{(\lambda)}_a$ acting on the Bergman space $\Aa_\lambda^2(D)$ is given by
\[
	T^{(\lambda)}_a = B_{D,\lambda} \circ M_a.
\]
With this notation, $a$ is called the symbol of $T^{(\lambda)}_a$.

Our main goal is to study commutative $C^*$-algebras generated by Toeplitz operators, with Lie theory and symplectic geometry as our main tools. We now proceed to describe some basic symplectic geometric properties of $\B^n$ and $D_n$. We refer to \cite{QSJFA,QVUnitBall1} for further details on the claims in the rest of this section.

The unit ball $\B^n$ carries the well known Bergman metric which is K\"ahler and invariant under the group of biholomorphisms of $\B^n$. The K\"ahler form for such structure is given by the following expression
\begin{align*}
    (\omega_{\B^n})_z &= i\sum_{j,k=1}^n \frac{(1-|z|^2)\delta_{jk} +
                \overline{z}_j z_k}{(1-|z|^2)^2}
                    \dif z_j \wedge \dif \overline{z}_k \\
        &= \frac{i}{(1-|z|^2)^2} \bigg(
            (1-|z|^2)\sum_{j=1}^n \dif z_j \wedge \dif \overline{z}_j +
                \sum_{j,k=1}^n \overline{z}_j z_k \dif z_j \wedge \dif \overline{z}_k\bigg),
\end{align*}
where $z \in \B^n$. Correspondingly, the Bergman metric on the Siegel domain $D_n$ is K\"ahler and invariant under its group of biholomorphisms. The associated K\"ahler form is now given by
\begin{multline*}
    (\omega_{D_n})_z = \frac{i}{(\im(z_n)-|z'|^2)^2}
        \bigg( (\im(z_n)-|z'|^2) \sum_{j=1}^{n-1}
                \dif z_j \wedge \dif \overline{z}_j
            + \sum_{j,k=1}^{n-1} \overline{z}_j z_k
                \dif z_j \wedge \dif \overline{z}_k  \\
          + \frac{1}{2i}
            \sum_{j=1}^{n-1} (\overline{z}_j \dif z_j\wedge \dif\overline{z}_n
                - z_j \dif z_n \wedge \dif \overline{z}_j)
          + \frac{1}{4} \dif z_n \wedge \dif \overline{z}_n
            \bigg),
\end{multline*}
where $z \in D_n$.

There exist exactly $n+2$ conjugacy classes of (connected) maximal Abelian subgroups (MASG for short) acting biholomorphically on either the unit ball $\B^n$ or the Siegel domain $D_n$. In this work we will be interested in only two of such actions. In the first place, we have the $\T^n$-action on $\B^n$ given by
\[
    t \cdot z = (t_1 z_1, \dots, t_n z_n),
\]
which will be referred as the quasi-elliptic action on the unit ball $\B^n$. Next, we have the $\T^{n-1}\times\R_+$-action on $D_n$ given by
\[
    (t',s) \cdot z = (s^\frac{1}{2} t' z', s z_n),
\]
that we will call the quasi-hyperbolic action on the Siegel domain $D_n$.

Some notions, computations and properties considered in this work are shared by these two actions. Hence, for simplicity we will denote by $(G,D)$ either of the pairs $(\T^n,\B^n)$ or $(\T^{n-1}\times \R_+, D_n)$ with the corresponding actions given above. This convention will be used in the rest of this work without further mention. As a first example of the similarity between these two cases, we note that the Lie algebra of $G$ is $\R^n$ for both choices.

We recall the definition of a moment map for the action of a MASG, which we will restrict to the two cases above. For every $X \in \R^n$, the $G$-action induces a smooth vector field on $D$ given by
\[
    X^\sharp_z =
        \frac{\dif}{\dif t}\sVert[2]_{s=0}
            \exp(sX)\cdot z
\]
for every $z \in D$, where $\cdot$ denotes the $G$-action on $D$.

\begin{definition}
    With the previous notation and conventions, a moment map for the $G$-action on $D$ is a smooth $G$-invariant function $\mu = \mu^G : D \rightarrow \R^n$ such that
    \[
        \dif \mu_X = \omega(X^\sharp, \cdot)
    \]
    for every $X \in \R^n$, where $\omega$ is the K\"ahler form of $D$ and $\mu_X : D \rightarrow \R$ is the smooth function given by
    \[
        \mu_X(z) = \langle \mu(z), X\rangle,
    \]
    for every $z \in D$, where $\langle\cdot,\cdot\rangle$ is the canonical inner product of $\R^n$.
\end{definition}

We note that we have used a definition slightly different from the one found in~\cite{QSJFA}. However, it is straightforward to show their equivalence.

We have proved in \cite{QSJFA} the existence, by explicit computation, of a moment map for every MASG of biholomorphisms of either $\B^n$ or $D_n$. The following formulas were obtained \cite{QSJFA} and will be used in the rest of this work.

For the $\T^n$-action on $\B^n$ a moment map is given by
\[
    \mu(z) = -\frac{1}{1 - |z|^2}
        (|z_1|^2, \dots, |z_n|^2),
\]
for every $z \in \B^n$. The $\T^{n-1}\times\R_+$-action on $D_n$ has a moment map given by
\[
    \mu(z) = -\frac{1}{2(\im(z_n) - |z'|^2)}
        (2|z_1|^2, \dots, 2|z_{n-1}|^2, \re(z_n)),
\]
for every $z \in D_n$.

\section{Group-moment coordinates}\label{sec:actionmomentcoord}
In this section, we will introduce some useful coordinates on $D$ obtained from the $G$-action and its moment map. These are basically obtained by considering the $G$-orbits with natural coordinates from $G$, and then using the corresponding moment map to complete a full set of smooth coordinates for $D$ on an open conull dense subset. This is done after a linear reparametrization of the moment map, which will simplify some of our computations later on. We observe that this does not change in any essential way the use of the moment map. In fact, it is easy to prove that any linear reparametrization of the moment map is also a moment map corresponding to a change of coordinates of the Lie algebra, in this case $\R^n$. We skip the proof of this fact since it will not be used in this work.

\subsection{Quasi-elliptic group-moment coordinates}
\label{subsec:quasi-elliptic-coordinates}
In this subsection, we introduce coordinates on $\B^n$ associated with the action and the moment map for the pair $(\T^n, \B^n)$. The coordinates will be defined on the open conull dense subset given by
\[
    \widehat{\B}^n =
        \{ z \in \B^n \mid z_j \not= 0 \text{ for all $j =1, \dots, n$ }\}.
\]
First, we consider the linear reparametrization $H : \widehat{\B}^n \rightarrow \R^n_+$ of the moment map given by
\[
    H(z) = \frac{1}{1 - |z|^2}
        (|z_1|^2, \dots, |z_n|^2).
\]
Note that $H$ is surjective. In a sense, the map $H$ provides half of the coordinates, and the other half will be obtained from the $\T^n$-action. To achieve this, we look for a smooth section of the map $H$, in other words, a smooth map $\sigma : \R^n_+ \rightarrow \widehat{\B}^n$ such that $H(\sigma(h)) = h$ for every $h \in \R^n_+$. In this work we will choose the map $\sigma : \R^n_+ \rightarrow \widehat{\B}^n$ given by
\[
    \sigma(h) = \frac{1}{(1 + \|h\|_1)^{\frac{1}{2}} }
    (h_1^{\frac{1}{2}}, \dots, h_n^{\frac{1}{2}}),
\]
for every $h \in \R^n_+$, where we will denote from now on $\|x\|_1 = |x_1| + \dots + |x_n|$ for every $x \in \R^n$. It is easily seen by direct computation that $\sigma$ is indeed a smooth section of $H$.

Next, the coordinates on $\widehat{\B}^n$ are obtained by letting $\T^n$ act on the values of the section $\sigma$. More precisely, we consider the map given by
\begin{align*}
	\kappa : \T^n \times \R^n_+ & \longrightarrow \widehat{\B}^n \\
	\kappa(t,h) = t \cdot \sigma(h) 
		&= \frac{1}{(1 + \|h\|_1)^{\frac{1}{2}} }
		(t_1 h_1^{\frac{1}{2}}, \dots, t_n h_n^{\frac{1}{2}}).
\end{align*}
The map $\kappa$ turns out be a diffeomorphism as we now prove. Let us consider the smooth map $\rho : \widehat{\B}^n \rightarrow \T^n$ given~by
\[
    \rho(z) = \bigg(\frac{z_1}{|z_1|}, \dots, \frac{z_n}{|z_n|} \bigg),
\]
for every $z \in \widehat{\B}^n$. This allows us to define the smooth map
\begin{align*}
	\tau : \widehat{\B}^n & \longrightarrow \T^n \times \R^n_+ \\
	\tau(z) = (\rho(z), H(z)) &= \bigg(\frac{z_1}{|z_1|}, \dots, \frac{z_n}{|z_n|},
	\frac{|z_1|^2}{1-|z|^2}, \dots, \frac{|z_n|^2}{1-|z|^2}\bigg),
\end{align*}
which now makes use of the moment map of the $\T^n$-action. A straightforward computation shows that $\kappa$ and $\tau$ are inverses of each other. Hence, $\kappa$ defines a set of coordinates for $\widehat{\B}^n$ and $\tau$ can be considered as the corresponding parametrization of $\widehat{\B}^n$ associated with such coordinates.

We now proceed to compute the Lebesgue measure, and its associated weighted measures, on $\B^n$ in terms of the coordinates given by $\kappa$ and $\tau$. Note that in polar coordinates we can write
\[
	\tau(z) = \bigg(t_1, \dots, t_n,
	\frac{r_1^2}{1-|r|^2}, \dots, \frac{r_n^2}{1-|r|^2}\bigg),
\]
where $z_j = r_j t_j$ and $r_j = |z_j|$, for every $j = 1, \dots, n$. As we have done above, we will denote by $h \in \R_+^n$ the elements in the target space of $H$. By considering the last $n$ coordinates of the previous expression for $\tau$, we obtain the change of variables $\varphi : \B^n \cap \R_+^n \rightarrow \R_+^n$ given by
\[
	h = \varphi(r) =
		\bigg(\frac{r_1^2}{1-|r|^2}, \dots, \frac{r_n^2}{1-|r|^2}\bigg).
\]
For this change of coordinates we have
\[
	1-|r|^2 = \frac{1}{1 + \|h\|_1}.
\]
Furthermore, a straightforward computation shows that 
\[
	\dif \varphi_r = \frac{2}{(1-|r|^2)^2}
		D(r) ((1 - |r|^2) I_n +r^\top r).
\]
Where, from now on, for every $z \in \C^m$ we will denote by $D(z)$ the diagonal matrix whose diagonal entries are given by the components of $z$. For a given $r \in \R^n_+$, let us choose $U \in \OO(n)$ such that $rU = |r|e_n$. Then, we have
\[
	U^\top((1 - |r|^2) I_n +r^\top r) U = (1-|r|^2)I_n + |r|^2 E_{n,n}
\]
where $E_{n,n}$ denotes the $n \times n$ matrix with zero entries everywhere except at position $(n,n)$. From this, it follows immediately that
\[
	\det(\dif \varphi_r) = \frac{2^n \prod_{j=1}^n r_j}{(1-|r|^2)^{n+1}}.
\]
Using polar coordinates and substituting these change of variables computations it follows that the Lebesgue measure on $\B^n$ in the group-moment coordinates of the quasi-elliptic case induces the following measure on $\T^n \times \R^n_+$
\[
	\dif \nu(t,h) 
		= \frac{1}{2^n} \frac{\dif t}{it}  \frac{\dif h}{(1 + 	
				\|h\|_1)^{n+1}},
\]
where we will denote $\frac{\dif t}{it}=\prod_{j=1}^{n} \frac{\dif t_j}{it_j}$. It also follows that, for every $\lambda > -1$, the weighted Lebesgue measure $v_\lambda$ induces the following measure on $\T^n \times \R^n_+$
\[
	\dif \nu_\lambda(t,h) 
		= \frac{c_\lambda}{2^n} \frac{\dif t}{it}  
				\frac{\dif h}{(1 + \|h\|_1)^{n+\lambda+1}}.
\]

\subsection{Quasi-hyperbolic group-moment coordinates}
\label{subsec:quasi-hyperbolic-coordinates}
For this subsection, we now consider coordinates on $D_n$ by using the action and the moment map associated with the pair $(\T^{n-1}\times \R_+, D_n)$. Such coordinates will now be defined on the open conull dense subset given by
\[
    \widehat{D}_n = \{ z \in D_n \mid z_j \not= 0 \text{ for all $j =1, \dots, n-1$ } \}.
\]
We use for this action the linear reparametrization $H : \widehat{D}_n \rightarrow \R^{n-1}_+ \times \R$ of the moment map given by
\[
    H(z) = \frac{1}{\im(z_n) - |z'|^2}
        (|z_1|^2, \dots, |z_{n-1}|^2, \re(z_n)).
\]
This map is surjective and it will provide half of the coordinates. Our goal is to define maps that allow us to obtain the rest of the coordinates from the $\T^{n-1}\times \R_+$-action. As before, we look for a smooth section of the map $H$. In this case, we consider the map $\sigma : \R^{n-1}_+ \times \R \rightarrow \widehat{D}_n$ defined by
\[
    \sigma(h) = (h_1^\frac{1}{2}, \dots, h_{n-1}^\frac{1}{2}, h_n + i(1 + \|h'\|_1)),
\]
for every $h \in \R^{n-1}_+ \times \R$. It is straightforward to verify that $\sigma$ is a section of~$H$.

As before, the coordinates on $\widehat{D}_n$ are obtained by letting the group $\T^{n-1} \times \R_+$ act on the values of $\sigma$. Hence, we consider the map given by
\begin{align*}
	\kappa : \T^{n-1} \times \R_+ \times \R^{n-1}_+ \times \R 
				&\longrightarrow \widehat{D}_n \\
	\kappa(t',s,h) &= (t',s)\cdot \sigma(h) \\
		&= \big(t_1 (sh_1)^\frac{1}{2}, \dots, 
		t_{n-1} (sh_{n-1})^\frac{1}{2}, 
		s(h_n + i(1 + \|h'\|_1))\big).
\end{align*}
We now prove that $\kappa$ is a diffeomorphism. For this, we consider the smooth map $\rho : \widehat{D}_n \rightarrow \T^{n-1} \times \R_+$ given by
\[
    \rho(z) =
    \bigg(\frac{z_1}{|z_1|}, \dots, \frac{z_{n-1}}{|z_{n-1}|}, \im(z_n) - |z'|^2
    \bigg),
\]
for every $z \in \widehat{D}_n$, which is used to define the smooth map
\begin{align*}
	\tau : \widehat{D}_n &\longrightarrow \T^{n-1} \times \R_+ 
			\times \R^{n-1}_+ \times \R	\\
	\tau(z) = (\rho(z), H(z)) 
		=&\; \bigg( \frac{z_1}{|z_1|}, \dots, \frac{z_{n-1}}{|z_{n-1}|},
				\im(z_n) - |z'|^2, \\
				&\frac{|z_1|^2}{\im(z_n) - |z'|^2}, \dots,
				\frac{|z_{n-1}|^2}{\im(z_n) - |z'|^2}, 
				\frac{\re(z_n)}{\im(z_n) - |z'|^2}
			\bigg).
\end{align*}
The smooth map $\tau$ now involves the moment map of the $\T^{n-1} \times \R_+$-action. A direct computation shows that $\kappa$ and $\tau$ are inverses of each other. It follows that the former defines a set of coordinates for $\widehat{D}_n$ and the latter the corresponding parametrization of $\widehat{D}_n$.

We now compute the Lebesgue measure, and the associated weighted measures, on $\widehat{D}_n$ in terms of the coordinates given by $\kappa$ and $\tau$.

Let us consider polar coordinates for $z' \in \C^{n-1}$ and rectangular coordinates for $z_n \in \C$, so that we can write
\[
	\tau(z) = \bigg(
	t', y - |r'|^2, \frac{r_1^2}{y - |r'|^2}, \dots, 
	\frac{r_1^2}{y - |r'|^2}, \frac{x}{y - |r'|^2} 
	\bigg),
\]
where $z_j = r_j t_j$ and $r_j = |z_j|$, for every $j = 1, \dots, n-1$, and $z_n = x + iy$. Next, we use a change of coordinates that uses the last $n+1$ components of the previous expression and thus we define the map $\varphi$ given by
\[
	(h,s) = \varphi(r',x,y) = \bigg(\frac{r_1^2}{y - |r'|^2}, \dots, 
	\frac{r_{n-1}^2}{y - |r'|^2}, \frac{x}{y - |r'|^2}, y - |r'|^2 
	\bigg).
\]
In particular, we have defined $s = y - |r'|^2$. A direct computation yields the differential
\begin{align*}
	\dif \varphi_{(r',x,y)} &=
	\left(
	\begin{array}{ccc}
		\frac{2}{(y - |r'|^2)^2}D(r') ((y-|r'|^2)I_n +r'^\top r') 
		& 0 & -\frac{1}{(y - |r'|^2)^2}D(r') r'^\top \\
		\frac{2x}{(y-|r'|^2)^2}r' & \frac{1}{y-|r'|^2} & -\frac{x}{(y-|r'|^2)^2} \\
		-2r' & 0 & 1
	\end{array} 
	\right) \\
	&=
	D\bigg(\frac{r'}{(y-|r'|^2)^2}, \frac{1}{y-|r'|^2},1\bigg)
	\left(
	\begin{array}{ccc}
		2((y-|r'|^2)I_n +r'^\top r') & 0 & r'^\top \\
		\frac{2x}{y-|r'|^2}r' & 1 & -\frac{x}{y-|r'|^2} \\
		-2r' & 0 & 1
	\end{array}
	\right)
\end{align*}
As before, for a given $r' \in \R^{n-1}_+$, we choose a matrix $U \in \OO(n-1)$ such that $r' U = |r'| e_{n-1}$, and conjugate the second factor in the last line by the block diagonal matrix $\mathrm{diag}(U^{-1},1,1)$ to obtain the determinant
\[
	\det(\dif \varphi_{(r',x,y)}) = 
		\frac{2^{n-1} \prod_{j=1}^{n-1} r_j}{(y - |r'|^2)^n}.
\]
Using the previous polar-rectangular coordinates and substituting our change of variable computations, we conclude that the Lebesgue measure on $D_n$ in the group-moment coordinates defined by the quasi-hyperbolic case induces the following measure on $\T^{n-1} \times \R_+ \times \R^{n-1}_+ \times \R$
\[
	\dif \nu(t',s,h) = 
		\frac{1}{2^{n-1}}\frac{\dif t'}{it'}s^n \dif s \dif h.
\]
It also follows that, for every $\lambda > -1$, the weighted Lebesgue measure $\widehat{v}_\lambda$ induces the following measure on $\T^{n-1} \times \R_+ \times \R^{n-1}_+ \times \R$
\[
	\dif \widehat{\nu}_\lambda(t',s,h) = 
			\frac{c_\lambda}{2^{n+1}}\frac{\dif t'}{it'}s^{n + \lambda} 
					\dif s \dif h.
\]

\section{Bergman spaces in group-moment coordinates}
\label{sec:Bergman-groupmoment-coord}
\subsection{Bergman spaces in quasi-elliptic group-moment coordinates}
\label{subsec:quasi-elliptic-Bergman}
In this subsection, we will use the notation and computations from subsection~\ref{subsec:quasi-elliptic-coordinates} without further mention.

Let us consider the operator
\begin{align*}
	U_0 : L^2(\B^n, v_\lambda) &\rightarrow 
			L^2(\T^n \times \R^n_+,\nu_\lambda)	\\
	U_0 f &= f \circ \kappa.
\end{align*}
By the definition of $\nu_\lambda$, this map is unitary with inverse given by $U_0^{-1}f = f\circ \tau$. Observe that we can write
\[
	L^2(\T^n \times \R^n_+,\nu_\lambda) =
		L^2(\T^n) \otimes L^2(\R^n_+, \widetilde{\nu}_\lambda),
\]
where we will denote from now on
\[
	\dif \widetilde{\nu}_\lambda(h) = 
		\frac{c_\lambda \dif h}{2^n(1+\|h\|_1)^{n+\lambda+1}}.
\]
Next, we consider the unitary operator
\begin{align*}
	U_1 : L^2(\T^n) \otimes L^2(\R^n_+,\widetilde{\nu}_\lambda) 
			&\longrightarrow
				\ell^2(\Z^n) \otimes 
				L^2(\R^n_+, \widetilde{\nu}_\lambda) 
				= \ell^2\big(\Z^n,
				L^2(\R^n_+, \widetilde{\nu}_\lambda)\big) \\
	U_1 &= \Fa_n \otimes I,
\end{align*}
where $\Fa_n$ is the discrete Fourier transform given by
\[
 	(\Fa_n f)(p) = 
 		\frac{1}{(2\pi)^\frac{n}{2}} 
 		\int_{\T^n} f(t)t^{-p} \frac{dt}{it},
\]
for every $p \in \Z^n$. The next result allows us to provide an alternative description of the Bergman space as well as a Bargmann-type transform.

\begin{theorem}\label{thm:Segal-Bargmann-QE}
	For every $n \geq 1$ and $\lambda > -1$, the map 
	\begin{align*}
		W_\lambda : \ell^2(\Z^n_+) &\longrightarrow
					\ell^2\big(\Z^n,
						L^2(\R^n_+,\widetilde{\nu}_\lambda)\big) \\
		W_\lambda(\psi)(p,h) &=
			\sqrt{\frac{(2\pi)^n (n + \lambda + 1)_{|p|}}%
				{p! (\lambda + 1)_n}} 
				\frac{h^\frac{p}{2}}%
					{(1 + \|h\|_1)^{\frac{|p|}{2}}}
			\psi(p)	
	\end{align*}
	where $(p,h) \in \Z^n_+ \times \R^n_+$, is an isometry that satisfies
	\begin{enumerate}
		\item $W_\lambda(\ell^2(\Z^n_+)) = (U_1 \circ U_0) 	
			(\Aa^2_\lambda(\B^n))$, and
		\item the adjoint $W_\lambda^*$ is given by
			\[
				(W_\lambda^* \varphi)(p) 
				= \sqrt{\frac{(\lambda + 1)_n (n + \lambda + 1)_{|p|}}%
					{(2\pi)^n p!}}
					\int_{\R^n_+} \frac{h^\frac{p}{2} \varphi(p,h) \dif h}%
					{(1 + \|h\|_1)^{\frac{|p|}{2}+n+\lambda+1}},
			\]
			for every $p \in \Z^n_+$, where $\varphi \in \ell^2\big(\Z^n,
			L^2(\R^n_+,\widetilde{\nu}_\lambda)\big)$.
	\end{enumerate}
	In particular, $R_\lambda = W_\lambda^* \circ U_1 \circ U_0 : L^2(\B^n, v_\lambda) \rightarrow \ell^2(\Z^n_+)$ is a surjective partial isometry with initial space $\Aa^2_\lambda(\B^n)$.
\end{theorem}
\begin{proof}
	We start by describing the space $(U_1\circ U_0)(\Aa^2_\lambda(\B^n))$. Observe that for every smooth function $f : \T^n \times \R^n_+ \rightarrow \C$ we have
	\begin{multline*}
		\bigg(\frac{\partial}{\partial \overline{z}_j} 
		(f\circ\tau)\bigg)(\kappa(t,h)) = \\
		= t_j\big(h_j(1+\|h\|_1)\big)^\frac{1}{2}
		\bigg(
		-\frac{t_j}{2 h_j}
		\frac{\partial f}{\partial t_j}(t,h) + 
		\frac{\partial f}{\partial h_j}(t,h) +
		\sum_{k=1}^n h_k \frac{\partial f}{\partial h_k}(t,h)
		\bigg),
	\end{multline*}
	for every $j =1, \dots, n$. This is obtained as a straightforward application of the chain rule. For every $j=1, \dots, n$, the vanishing of the previous equations is equivalent to the following system of equations
	\[
		\frac{t_j}{2h_j}\frac{\partial f}{\partial t_j} 
		= \frac{\partial f}{\partial h_j}
		+ \sum_{k=1}^n h_k \frac{\partial f}{\partial h_k},
	\]
	where $j = 1, \dots, n$ and $f : \T^n \times \R^n_+ \rightarrow \C$ is smooth. It follows that $U_0(\Aa^2_\lambda(\B^n))$ is the space of smooth functions $f \in L^2(\T^n \times \R^n_+, \nu_\lambda)$ that satisfy the previous system of equations. Applying the operator $U_1$ to the latter we conclude that $(U_1 \circ U_0)(\Aa^2_\lambda(\B^n))$ consists of the functions $\varphi \in \ell^2\big(\Z^n, L^2(\R^n_+, \widetilde{\nu}_\lambda)\big)$ that satisfy the system of equations
	\begin{equation} 
		\label{eq:CR-QE-equations}
		\frac{p_j}{2h_j} \varphi(p,h) = \frac{\partial \varphi}{\partial h_j}(p,h)
		+ \sum_{k=1}^n h_k \frac{\partial \varphi}{\partial h_k}(p,h),
	\end{equation}
	at every $(p,h) \in \Z^n \times \R^n_+$ and for every $j = 1, \dots, n$. It is straightforward to check that the general solution of \eqref{eq:CR-QE-equations} is given by
	\begin{equation}
		\label{eq:CR-QE-solutions}
		\varphi(p,h) = \frac{h^{\frac{p}{2}}}{(1 + \|h\|_1)^{\frac{|p|}{2}}}
					\psi(p),
	\end{equation}
	for every $(p,h) \in \Z^n \times \R^n_+$, where $\psi : \Z^n \rightarrow \C$ is arbitrarily given. In particular, the functions given by the expression of $W_\lambda(\psi)$ in the statement are solutions to the system of equations \eqref{eq:CR-QE-equations} for any $\psi : \Z^n_+ \rightarrow \C$.
	
	We now consider the condition $\varphi \in \ell^2\big(\Z^n, L^2(\R^n_+, \widetilde{\nu}_\lambda)\big)$. For $\varphi$ of the form \eqref{eq:CR-QE-solutions} this is equivalent to the requirement that
	\[
		\|\varphi\|^2 = 
			\frac{c_\lambda}{2^n} \sum_{p \in \Z^n} |\psi(p)|^2
				\Bigg(
				\int_{\R^n_+} 
					\frac{h^p \dif h}{(1 + \|h\|_1)^{|p|+n+\lambda+1}}
				\Bigg) < \infty.
	\]
	It is well known (see \cite[4.683(3)]{GR}) that the integral inside parentheses above is finite exactly for $p \in \Z^n_+$ and in that case it has the following value
	\[
		\int_{\R^n_+} 
				\frac{h^p \dif h}{(1 + \|h\|_1)^{|p|+n+\lambda+1}} 
				= \frac{p! \Gamma(n + \lambda + 1)}%
						{\Gamma(|p| + n + \lambda + 1)}.
	\]
	Hence, the $L^2$ condition for solutions of the form \eqref{eq:CR-QE-solutions} requires to consider only functions $\psi \in \ell^2(\Z^n_+)$. Furthermore, we also conclude that multiplying a given $\psi \in \ell^2(\Z^n_+)$ in the solution \eqref{eq:CR-QE-solutions} by the factor
	\begin{align*}
		\sqrt{\frac{2^n}{c_\lambda}
		\frac{\Gamma(|p| + n + \lambda + 1)}%
		{p! \Gamma(n + \lambda + 1)}} 
		&= \sqrt{\frac{(2\pi)^n\Gamma(|p| + n + \lambda + 1)}%
				{p! (\lambda + 1)_n \Gamma(n+\lambda+1)}} \\
		&= \sqrt{\frac{(2\pi)^n(n + \lambda + 1)_{|p|}}%
				{p! (\lambda + 1)_n}}
	\end{align*}
	we obtain an isometry mapping $\ell(\Z^n_+)$ onto $(U_1 \circ U_0)(\Aa^2_\lambda(\B^n))$. But this yields precisely the definition of $W_\lambda$, and so we thus conclude that it is an isometry. Hence, (1) from our statement has been proved. The formula for $W_\lambda^*$ stated in (2) is easily obtained by considering the identity $\langle W_\lambda^* \varphi, \psi \rangle = \langle \varphi, W_\lambda \psi \rangle$ which holds for $\varphi \in \ell^2\big(\Z^n, L^2(\R^n_+,\widetilde{\nu}_\lambda)\big)$ and $\psi \in \ell^2(\Z^n_+)$.	
	
	The last claim now follows from the basic properties of partial isometries.
\end{proof}

\subsection{Bergman spaces in quasi-hyperbolic group-moment coordinates}
\label{subsec:quasi-hyperbolic-Bergman}
We will now use the notation and computations from subsection~\ref{subsec:quasi-hyperbolic-coordinates}.

Let us consider the operator
\begin{align*}
	U_0 : L^2(D_n, \widehat{v}_\lambda) &\longrightarrow 
		L^2(\T^{n-1} \times \R_+ \times \R^{n-1}_+ \times \R,
		\widehat{\nu}_\lambda)	\\
	U_0 f &= f \circ \kappa,
\end{align*}
which is unitary with inverse given by $U_0^{-1}f = f\circ \tau$ by the choice of $\widehat{\nu}_\lambda$. As before, we can use the identification
\begin{multline*}
	L^2(\T^{n-1} \times \R_+ \times \R^{n-1}_+ \times \R, 	
		\widehat{\nu}_\lambda) \\ 
		= L^2(\T^{n-1}) 
			\otimes L^2(\R_+, s^{n+\lambda}\dif s) 
			\otimes L^2\Big(\R^{n-1}_+ 
				\times \R, \frac{c_\lambda}{2^{n+1}} \dif h\Big). 
\end{multline*}
We now consider the unitary operator $U_1$ that maps the latter
onto 
\begin{multline*}
	L^2\Big(\Z^{n-1} \times \R \times \R^{n-1}_+ \times \R, 	
		\frac{c_\lambda}{2^{n+1}} \dif p' \dif \xi \dif h\Big)  \\
		= \ell^2(\Z^{n-1}) \otimes L^2(\R) \otimes L^2(\R^{n-1}_+ \times \R, 	
			\frac{c_\lambda}{2^{n+1}} \dif h)	
\end{multline*}
given by
\[
		U_1 = \Fa_{n-1} \otimes \Ma \otimes I,
\]
where $\Fa_{n-1}$ is the discrete Fourier transform defined as in subsection~\ref{subsec:quasi-elliptic-coordinates} and $\Ma$ is the following Mellin transform 
\begin{align*}
	\Ma : L^2(\R_+,s^{n+\lambda} \dif s) 
		&\longrightarrow L^2(\R) \\
	(\Ma f)(\xi) &= \frac{1}{\sqrt{2\pi}}
		\int_0^\infty f(s) s^{\frac{n + \lambda - 1}{2} -i\xi} \dif s.
\end{align*}
We next consider a further unitary map that will require the introduction of some notation.

Let us denote from now on by $F : \Z^{n-1}_+ \times \R \times \R^{n-1}_+ \times \R \rightarrow \C$ the function given by the following expression
\[
	F(p',\xi,h',u) 
	= (1 + \|h'\|_1)^{i\xi - \frac{n + \lambda}{2}}
		(i + u)^{i\xi - \frac{|p'| + n + \lambda + 1}{2}}.
\]
Note that we will consider the branch of the logarithm for which the argument lies in the interval $(0,2 \pi)$. With this convention, it is easy to prove that
\[
	|F(p',\xi,h',u)|^2 
	= \frac{e^{-2\xi \arccot(u)}}%
			{(1 + \|h'\|_1)^{n + \lambda}%
				(1 + u^2)^{\frac{|p'| + n + \lambda + 1}{2}}},
\]
for every $(p',\xi,h',u) \in \Z^{n-1}_+ \times \R \times \R^{n-1}_+ \times \R$. Here, we have used the branch of $\arccot$ defined on $\R$ that takes values in $(0,\pi)$.

On the other hand, we will consider the following weighted measure on $\R$
\begin{equation}\label{eq:Romanovski-Routh}
	\frac{e^{-\alpha\arccot(u)}}{(1 + u^2)^{1-\beta}} \dif u.
\end{equation}
As noted in \cite{RomanMartinez}, the weight involved is precisely the one used to defined the so-called Romanovski-Routh polynomials through the Rodrigues formula. We will denote by $V(\alpha, \beta)$ the corresponding total measure. In other words, we have
\[
	V(\alpha,\beta) 
		= \int_\R 
			\frac{e^{-\alpha\arccot(u)}}{(1 + u^2)^{1-\beta}} \dif u,
\]
which is clearly finite whenever $\beta < 0$.

With the previous notation, we consider the operator $U_2$ that maps the space 
\[
	L^2\Big(\Z^{n-1} \times \R \times \R^{n-1}_+ \times \R, \frac{c_\lambda}{2^{n+1}} \dif p' \dif \xi \dif h\Big)
\]
to the space 
\begin{multline*}
	L^2 \Big(
		\Z^{n-1} \times \R \times \R^{n-1}_+ \times \R, 
		\frac{c_\lambda}{2^{n+1}} |F(p',\xi,h',u)|^2
			\dif p' \dif \xi \dif h' \dif u 
		\Big)   \\
	= \bigoplus_{p' \in \Z^{n-1}} 
		L^2 \Big(
			\R \times \R^{n-1}_+ \times \R, 
			\frac{c_\lambda}%
			{2^{n+1}} |F(p',\xi,h',u)|^2 \dif \xi \dif h' \dif u 
		\Big) 
\end{multline*}
given by
\[
	(U_2 \varphi)(p',\xi,h',u) 
		= \varphi\big(p',\xi,h', (1 + \|h'\|_1)u\big)
			\frac{(1 + \|h'\|_1)^\frac{1}{2}}{F(p',\xi,h',u)},
\]
for every $(p',\xi,h',u) \in \Z^{n-1}_+ \times \R \times \R^{n-1}_+ \times \R$. A straightforward computation shows that $U_2$ is a well defined unitary operator.

The next result yields our alternative description of the Bergman space, as well as a Bargmann-type transform, in the current quasi-hyperbolic case.

\begin{theorem}\label{thm:Segal-Bargmann-QH}
	For every $n \geq 1$ and $\lambda > -1$, the operator $W_\lambda$ defined on the Hilbert space $\ell^2\big(\Z^{n-1}_+,L^2(\R)\big)$ and taking values into the Hilbert space
	\[
		L^2 \Big(
		\Z^{n-1} \times \R \times \R^{n-1}_+ \times \R, 
		\frac{c_\lambda}{2^{n+1}} |F(p',\xi,h',u)|^2
		\dif p' \dif \xi \dif h' \dif u 
		\Big) 
	\]
	given by 
	\[
		(W_\lambda \psi)(p',\xi,h',u)
		= \sqrt{\frac{2(2\pi)^n (n+\lambda)_{|p'|}}%
			{p'!(\lambda+1)_n V\big(
				2\xi, -\frac{|p'|+n+\lambda-1}{2}
				\big)}}
			\frac{\big(h'\big)^\frac{p'}{2}}%
				{(1 + \|h'\|_1)^\frac{|p'|}{2}} \psi(p',\xi)
	\]
	where $(p',\xi) \in \Z^{n-1}_+ \times \R$, is an isometry that satisfies
	\begin{enumerate}
		\item $W_\lambda\big(\ell^2\big(\Z^{n-1}_+,L^2(\R)\big)\big)
		= (U_2 \circ U_1 \circ U_0)(\Aa^2_\lambda(D_n))$, and
		\item the adjoint $W_\lambda^*$ is given by
		\begin{align*}
			(&W_\lambda^* \varphi)(p',\xi) 
				= \sqrt{\frac{(\lambda+1)_n (n+\lambda)_{|p'|}}%
					{2(2\pi)^n  p'!
						V\big(2\xi, -\frac{|p'|+n+\lambda-1}{2}
						\big)}} \times \\
				&\times \int_{\R^{n-1}_+ \times \R}
					\frac{\big(h'\big)^\frac{p'}{2}}%
						{(1 + \|h'\|_1)^{\frac{|p'|}{2}+n+\lambda}}
					\frac{e^{-2\xi \arccot(u)}}%
						{(1 + u^2)^{\frac{|p'| + n + \lambda + 1}{2}}}
					\varphi(p',\xi,h',u) \dif h' \dif u,
		\end{align*}
		for every $p' \in \Z^{n-1}_+$, $\xi \in \R$, where $\varphi$ belongs to the second listed Hilbert space above.
	\end{enumerate}
	In particular, $R_\lambda = W_\lambda^* \circ U_2 \circ U_1 \circ U_0 : L^2(D_n, \widehat{v}_\lambda) \rightarrow \ell^2\big(\Z^{n-1}_+,L^2(\R)\big)$ is a surjective partial isometry with initial space $\Aa^2_\lambda(D_n)$.
\end{theorem}
\begin{proof}
	We start with a description of the space $(U_1 \circ U_0)(\Aa^2_\lambda(D_n))$. First observe that for every smooth function $f$ defined on $\T^{n-1} \times \R_+ \times \R^{n-1}_+ \times \R$ we have
	\begin{align*}
		\bigg(\frac{\partial}{\partial \overline{z}_j} 
			(f\circ\tau)\bigg)(\kappa(t',s,h)) =&\; 
			t_j \sqrt{\frac{h_j}{s}}
			\Bigg(
			\sum_{k=1}^n h_k \frac{\partial f}{\partial h_k}(t',s,h)
			+ \frac{\partial f}{\partial h_j}(t',s,h) \\
		&- s \frac{\partial f}{\partial s}(t',s,h)
			- \frac{t_j}{2h_j}\frac{\partial f}{\partial t_j}(t',s,h)
			\Bigg), \\
			\bigg(\frac{\partial}{\partial \overline{z}_n} 
			(f\circ\tau)\bigg)(\kappa(t',s,h)) =&\; 
			\frac{1}{2is}
			\Bigg(
			\sum_{k=1}^n h_k \frac{\partial f}{\partial h_k}(t',s,h)
			\\
		&- s \frac{\partial f}{\partial s}(t',s,h)
			+ i \frac{\partial f}{\partial h_n}(t',s,h)
		\Bigg),	
	\end{align*}
	where $j = 1, \dots, n-1$. This is an immediate consequence of the chain rule. The vanishing of the previous expressions is equivalent to the following system of equations
	\begin{align*}
		\frac{\partial f}{\partial h_j}
		&= \frac{t_j}{2h_j}\frac{\partial f}{\partial t_j}
			+ i \frac{\partial f}{\partial h_n}, \\
			i \frac{\partial f}{\partial h_n}
		&= s \frac{\partial f}{\partial s}
			- \sum_{k=1}^n h_k \frac{\partial f}{\partial h_k},
	\end{align*}
	where $j = 1, \dots, n-1$ and $f : \T^{n-1} \times \R_+ \times \R_+^{n-1} \times \R \rightarrow \C$ is smooth. Hence, $U_0(\Aa^2_\lambda(D_n))$ consists of the functions $f$ that belong to $L^2(\T^{n-1} \times \R_+ \times \R^{n-1}_+ \times \R, \widehat{\nu}_\lambda)$ and that satisfy the latter set of equations. Next, we apply the operator $U_1$ to the latter and some algebraic manipulations to obtain that $(U_1 \circ U_0)(\Aa^2_\lambda(D_n))$ is the subspace of functions $\varphi \in L^2\Big(\Z^{n-1} \times \R \times \R^{n-1}_+ \times \R, \frac{c_\lambda}{2^{n+1}} \dif p' \dif \xi \dif h\Big)$ that satisfy the next system of equations
	\begin{align}
		\label{eq:CR-QH-discrete}
		\frac{\partial \varphi}{\partial h_j}(p',\xi,h)
		&= \bigg(\frac{p_j}{2h_j}
			+ i\frac{i\xi - \frac{|p'| + n + \lambda + 1}{2}}%
			{i(1 + \|h'\|_1) + h_n}
			\bigg)\varphi(p',\xi,h)  \\
				\frac{\partial \varphi}{\partial h_n}(p',\xi,h) 
		&= \frac{i\xi - \frac{|p'| + n + \lambda + 1}{2}}%
			{i(1 + \|h'\|_1) + h_n}
			\varphi(p',\xi,h), \notag
	\end{align}
	at every $(p',\xi,h) \in \Z^{n-1} \times \R \times \R_+^{n-1} \times \R$ and where $j =1, \dots, n-1$. The general solution of the system of equations \eqref{eq:CR-QH-discrete} is easily seen to be given by smooth functions $\varphi : \Z^{n-1} \times \R \times \R_+^{n-1} \times \R \rightarrow \C$ of the form
	\begin{equation}
		\label{eq:CR-QH-solutions-1}
		\varphi(p',\xi,h) = \psi(p',\xi) 
			\big(h'\big)^{\frac{p'}{2}} 
			\big(
				i(1 + \|h'\|_1) + h_n
			\big)^{i\xi - \frac{|p'| + n + \lambda + 1}{2}}
	\end{equation}
	for every $(p',\xi,h) \in \Z^{n-1} \times \R \times \R_+^{n-1} \times \R$, where $\psi : \Z^{n-1} \times \R \rightarrow \C$ is any smooth function. Without loss of generality, we can rewrite these solutions as
	\begin{equation}
		\label{eq:CR-QH-solutions-2}
		\varphi(p',\xi,h) 
			= \frac{\psi(p',\xi)}{V\Big(
				2\xi, -\frac{|p'|+n+\lambda-1}{2}
				\Big)^\frac{1}{2}}
			\frac{\big(h'\big)^{\frac{p'}{2}}}%
				{(1 + \|h'\|_1)^\frac{|p'|}{2}} 
			\frac{F\Big(p',\xi,h',\frac{h_n}{1 + \|h'\|_1}\Big)}%
				{(1 + \|h'\|_1)^\frac{1}{2}}
	\end{equation}
	where the smooth function $\psi : \Z^{n-1} \times \R \rightarrow \C$ and $(p',\xi,h)$ are as above. This uses the definitions of $F$, $V$ and our choice of the branch of $\log$. It remains to consider the $L^2$ condition on these solutions. For this, we apply the expression used to define $U_2$ to functions $\varphi$ as in \eqref{eq:CR-QH-solutions-2}, which yields the following
	\begin{equation}
		\label{eq:CR-QH-solutions-3}
		(U_2\varphi)(p',\xi,h',u) 
			= \frac{\psi(p',\xi)}{V\Big(
				2\xi, -\frac{|p'|+n+\lambda-1}{2}
					\Big)^\frac{1}{2}}
				\frac{\big(h'\big)^{\frac{p'}{2}}}%
					{(1 + \|h'\|_1)^\frac{|p'|}{2}},
	\end{equation}
	where $(p',\xi,h',u) \in \Z^{n-1} \times \R \times \R^{n-1}_+ \times \R$. Hence, we consider $L^2$-integrability by computing
	\begin{align*}
		\|U_2\varphi\|^2 
			=&\; \frac{c_\lambda}{2^{n+1}%
				V\Big(
				2\xi, -\frac{|p'|+n+\lambda-1}{2}
				\Big)} \times \\
			&\times \sum_{p' \in \Z^{n-1}} 
				\int_{\R \times \R^{n-1}_+ \times \R}
					|\psi(p',\xi)|^2 
					\frac{\big(h'\big)^{p'}}%
					{(1 + \|h'\|_1)^{|p'|}}
					|F(p',\xi,h',u)|^2 \dif \xi \dif h' \dif u \\
			=&\; \frac{c_\lambda}{2^{n+1}%
				V\Big(
				2\xi, -\frac{|p'|+n+\lambda-1}{2}
				\Big)} \times \\
			&\times \sum_{p' \in \Z^{n-1}} 
				\bigg(
					\int_{\R^{n-1}_+}	\frac{\big(h'\big)^{p'}}%
					{(1 + \|h'\|_1)^{|p'|}} \dif h'
				\bigg) \times \\
			&\times
				\bigg(
					\int_\R |\psi(p',\xi)|^2
					\Big(
						\int_\R 
							\frac{e^{-2\xi\arccot(u)}}%
							{(1 + u^2)^\frac{|p'|+n+\lambda+1}%
								{2}} \dif u
					\Big)\dif \xi
				\bigg)   \\
			=&\; \frac{c_\lambda}{2^{n+1}} 
			\sum_{p' \in \Z^{n-1}} 
				\bigg(
					\int_{\R^{n-1}_+}	\frac{\big(h'\big)^{p'}}%
						{(1 + \|h'\|_1)^{|p'|}} \dif h'
				\bigg) 
				\bigg(
					\int_\R |\psi(p',\xi)|^2 \dif \xi
				\bigg),
	\end{align*}
	where we have used the remarks found above for the functions $F$ and $V$. As in the proof of Theorem~\ref{thm:Segal-Bargmann-QE}, we conclude that the previous expression is finite precisely when $p' \in \Z^{n-1}_+$, and in that case the value is given by
	\[
		\|U_2\varphi\|^2 
		= \frac{c_\lambda}{2^{n+1}}
			\sum_{p' \in \Z^{n-1}_+} 
			\frac{p'! \Gamma(n + \lambda)}{\Gamma(|p'| + n + \lambda)}
			\int_\R |\psi(p',\xi)|^2 \dif \xi.
	\]
	After some algebraic manipulations, this clearly implies that $W_\lambda$ is indeed an isometry with image $(U_2 \circ U_1 \circ U_0)(\Aa^2_\lambda(D_n))$. Thus, we have proved (1) from our statement. The formula for $W_\lambda^*$ in (2) is obtained easily from the identity $\langle W_\lambda^* \varphi, \psi \rangle = \langle \varphi, W_\lambda \psi \rangle$ which holds for $\varphi, \psi$ in the corresponding spaces.	
	
	The last claim follows from the properties of partial isometries.
\end{proof}

\section{Toeplitz operators with moment map symbols}
\label{sec:Toeplitz-momentmapsymbols}
From the notation introduced in section~\ref{sec:analysisgeometry}, let us recall that $(G,D)$ denotes either of the pairs $(\T^n,\B^n)$ or $(\T^{n-1}\times\R_+, D_n)$ with the associated quasi-elliptic and quasi-hyperbolic actions, respectively. Following \cite{QSJFA} we have the next definition.

\begin{definition}\label{def:momentmap-symbol}
	For $(G, D)$ as above and $\mu = \mu^G$ a moment map for the $G$-action on $D$, a symbol $a \in L^\infty(D)$ is called a moment map symbol for such action if there is a measurable function $f$ defined on the image of $\mu$ such that $a = f \circ \mu$. The collection of all moment map symbols is denoted by $L^\infty(D)^{\mu}$. A given $a \in L^\infty(D)$ will be called quasi-elliptic moment map symbol or quasi-hyperbolic moment map symbol according to whether either of the pairs $(\T^n,\B^n)$ or $(\T^{n-1}\times\R_+, D_n)$, respectively, is under consideration.
\end{definition}

The next result is a particular case of properties for MASG's proved in \cite{QSJFA}.

\begin{proposition}\label{prop:momentmap-vs-invariant}
	For $(G,D)$ as above a symbol $a \in L^\infty(D)$ is a moment map symbol if and only if it is $G$-invariant, in other words when for every $g \in G$ one has
	\[
		a(g\cdot z) = a(z)
	\]
	for almost every $z \in D$.
\end{proposition}

\begin{remark}\label{rmk:momentmap-vs-invariant}
	The space of $G$-invariant essentially bounded measurable symbols on $D$ is denoted by $L^\infty(D)^G$. Hence, the previous result states that
	\[
	L^\infty(D)^{\mu^G} = L^\infty(D)^G,
	\]
	for the pairs $(G, D)$ introduced above. We will use this property in the rest of this work without further mention. We also observe that this property holds for arbitrary MASG in the case of the unit ball and the Siegel domain (see~\cite{QSJFA} for further details).
	
	From the definition of group-moment coordinates in section~\ref{sec:actionmomentcoord} given by the diffeomorphisms $\kappa$ and $\tau$, it follows that a symbol $a$ is a moment map symbol if and only if it only depends on the moment part of such coordinates. More precisely, we recall that the coordinate map is given by $\tau = (\rho, H)$ and, with this notation, $a$ is a moment map symbol if and only if there is a measurable function $f$, defined on the image of $H$, such that $a = f \circ H$. This is a consequence of the fact that $H$ was chosen as a linear reparametrization of the corresponding moment map. This fact and notation will also be used below without further mention.
\end{remark}

In what follows of this section we will obtain diagonalizing spectral integral formulas for Toeplitz operators with moment map symbols corresponding to the quasi-elliptic and quasi-hyperbolic cases.

\subsection{Quasi-elliptic moment map symbols}
\label{subsec:quasi-elliptic-symbols}
Let us consider moment map symbols for the quasi-elliptic $\T^n$-action on $\B^n$. Hence, in this subsection we will consider the notation and results from subsections~\ref{subsec:quasi-elliptic-coordinates} and \ref{subsec:quasi-elliptic-Bergman}. Recall that from Remark~\ref{rmk:momentmap-vs-invariant}, quasi-elliptic moment map symbols are precisely those of the form $a = f \circ H$, where $f$ is essentially bounded and measurable defined on $\R^n_+$.

\begin{theorem}\label{thm:quasi-elliptic-spectral}
	Let $a = f \circ H \in L^\infty(\B^n)$ be a quasi-elliptic moment map symbol on $\B^n$. Then, for every $\lambda > -1$, the Toeplitz operator $T^{(\lambda)}_a$ acting on $\Aa^2_\lambda(\B^n)$ is unitarily equivalent to a multiplier operator. More precisely, for $R_\lambda$ the operator considered in Theorem~\ref{thm:Segal-Bargmann-QE} we have
	\[
		R_\lambda \circ T^{(\lambda)}_a \circ R_\lambda^* 
			= M_{\gamma_{a,\lambda}},
	\]
	where $M_{\gamma_{a,\lambda}}$ is the multiplier operator acting on $\ell^2(\Z^n_+)$ and associated to the function $\gamma_{a,\lambda} : \Z^n_+ \rightarrow \C$ given by
	\[
		\gamma_{a,\lambda}(p) 
			= \frac{(n + \lambda + 1)_{|p|}}{p!}
				\int_{\R^n_+} \frac{h^p}%
					{(1 + \|h\|_1)^{|p|+n+\lambda+1}}
					f(h) \dif h,
	\]
	for every $p \in \Z^n_+$.
\end{theorem}
\begin{proof}
	Recall from Theorem~\ref{thm:Segal-Bargmann-QE} that $R_\lambda = W_\lambda^* \circ U_1 \circ U_0$ where, in the rest of this proof, we follow the notation from subsection~\ref{subsec:quasi-elliptic-Bergman}. First note that the last claim of Theorem~\ref{thm:Segal-Bargmann-QE} implies that $R_\lambda^* R_\lambda = B_{\B^n,\lambda}$, the Bergman projection onto $\Aa^2_\lambda(\B^n)$, and that $R_\lambda R_\lambda^*$ is the identity operator acting on $\ell^2(\Z^n_+)$. Hence, we can compute
	\begin{align*}
		R_\lambda \circ T^{(\lambda)}_a \circ R_\lambda^*
			&= R_\lambda B_{\B^n,\lambda} 
						M_a B_{\B^n,\lambda} R_\lambda^* \\
			&= R_\lambda R_\lambda^* R_\lambda 
						M_a R_\lambda^* R_\lambda R_\lambda^* \\
			&= R_\lambda M_a R_\lambda^* \\
			&= W_\lambda^* U_1 U_0 M_a U_0^* U_1^* W_\lambda \\
			&= W_\lambda^* U_1 M_{a \circ \kappa} U_1^* W_\lambda,
	\end{align*}
	where the last line follows from the definition of $U_0$. From subsection~\ref{subsec:quasi-elliptic-coordinates}, the function $\tau$ is inverse of $\kappa$ and has its last $n$ components given by $H$. Hence, the property $a = f \circ H$ implies that
	\[
		a\circ \kappa(t,h) = f(h),
	\]
	for all $t \in \T^n$ and $h \in \R^n_+$. By its definition, $U_1$  commutes with $M_{a \circ \kappa} = M_f$ and so we arrive to
	\[
		R_\lambda \circ T^{(\lambda)}_a \circ R_\lambda^*
			= W_\lambda^* \circ M_f \circ W_\lambda.
	\]
	The formulas from Theorem~\ref{thm:Segal-Bargmann-QE} imply that the last operator satisfies, for every $\psi \in \ell^2(\Z^n_+)$, the following
	\[
		(W_\lambda^* \circ M_f \circ W_\lambda)(\psi)(p)
			= \frac{(n + \lambda + 1)_{|p|}}{p!}
				\Bigg(\int_{\R^n_+} \frac{h^p f(h) \dif h}%
					{(1 + \|h\|_1)^{|p|+n+\lambda+1}}\Bigg) \psi(p),
	\]
	for every $p \in \Z^n_+$. This yields the required conclusion.
\end{proof}

\begin{remark}\label{rmk:quasi-elliptic-spectral}
	From the integral identities considered in the proof of Theorem~\ref{thm:Segal-Bargmann-QE} we can rewrite $\gamma_{a,\lambda}$ from Theorem~\ref{thm:quasi-elliptic-spectral} as 
	\begin{equation}\label{eq:quasi-elliptic-spectral}
		\gamma_{a,\lambda}(p) 
		= \frac{\displaystyle\int_{\R^n_+} \frac{h^p}%
			{(1 + \|h\|_1)^{|p|+n+\lambda+1}} f(h) \dif h}%
			{\displaystyle\int_{\R^n_+} \frac{h^p}%
				{(1 + \|h\|_1)^{|p|+n+\lambda+1}}  \dif h}
	\end{equation}
	for every $p \in \Z^n_+$. This highlights the fact that $\gamma_{a,\lambda}$ is a sort of normalized weighted special integral with weight $f$. Of course, the function $f$ comes from the symbol $a$ through the identity $a = f \circ H$.
\end{remark}

\subsection{Quasi-hyperbolic moment map symbols}
\label{subsec:quasi-hyperbolic-symbols}
We now consider moment map symbols for the quasi-hyperbolic $\T^{n-1} \times \R_+$-action on $D_n$. In this case, we consider the notation and results from subsections~\ref{subsec:quasi-hyperbolic-coordinates} and \ref{subsec:quasi-hyperbolic-Bergman}. Again, Remark~\ref{rmk:momentmap-vs-invariant} implies that the quasi-hyperbolic moment map symbols are precisely those of the form $a = f \circ H$, where $f$ is an essentially bounded and measurable function defined on $\R^{n-1}_+ \times \R$. Furthermore, in this case it will be useful to consider the reparametrization of $H$ given by the following function
\begin{align}\label{eq:H-reparam}
	\widetilde{H} : D_n &\longrightarrow \R_+^{n-1} \times \R \\
	\widetilde{H}(z) &= (\widetilde{H}_0(z), \widetilde{H}_n(z)), \notag
\end{align}
where we denote
\begin{align}
	\widetilde{H}_0(z) &= \bigg(
			\frac{|z_1|^2}{\im(z_n) - |z'|^2}, \dots,
			\frac{|z_{n-1}|^2}{\im(z_n) - |z'|^2} \bigg)    
			\label{eq:widetildeH0}   \\
	\widetilde{H}_n(z) &= \frac{\re(z_n)}{\im(z_n)},
			\notag
\end{align}
for every $z \in D_n$. In particular, $H$ and $\widetilde{H}$ have the same first $n-1$ components. Moreover, we have
\begin{equation}\label{eq:widetildeHn-formula}
	\widetilde{H}_n(z) = \frac{H_n(z)}{1 + \|\widetilde{H}_0(z)\|_1}
		= \frac{\re(z_n)}{\im(z_n)},
\end{equation}
for every $z \in D_n$. These remarks and Proposition~\ref{prop:momentmap-vs-invariant} imply the next result.

\begin{lemma}\label{lem:H-vs-widetildeH-quasi-hyperbolic}
	A symbol $a \in L^\infty(D_n)$ is a quasi-hyperbolic moment map symbol if and only if any of the following equivalent conditions is satisfied.
	\begin{enumerate}
		\item $a$ is $\T^{n-1} \times \R_+$-invariant.
		\item $a = f \circ H$ for some measurable function $f$.
		\item $a = f \circ \widetilde{H}$ for some measurable function $f$.
	\end{enumerate}
\end{lemma}

Hence, the next result considers the most general sort of quasi-hyperbolic moment map symbol.

\begin{theorem}\label{thm:quasi-hyperbolic-spectral}
	Let $a = f \circ \widetilde{H} \in L^\infty(D_n)$ be a quasi-hyperbolic moment map symbol on $D_n$, where $f : \R_+^{n-1} \times \R \rightarrow \C$ is an essentially bounded measurable function. Then, for every $\lambda > -1$, the Toeplitz operator $T^{(\lambda)}_a$ acting on $\Aa^2_\lambda(D_n)$ is unitarily equivalent to a multiplier operator. More precisely, for $R_\lambda$ the operator considered in Theorem~\ref{thm:Segal-Bargmann-QH} we have
	\[
		R_\lambda \circ T^{(\lambda)}_a \circ R_\lambda^*
			= M_{\gamma_{a,\lambda}},
	\]
	where $M_{\gamma_{a,\lambda}}$ is the multiplier operator acting on $\ell^2\big(\Z^{n-1}_+,L^2(\R)\big)$ associated to the function $\gamma_{a,\lambda} : \Z^{n-1}_+ \times \R \rightarrow \C$ given by
	\begin{align*}
		\gamma_{a,\lambda}&(p',\xi) = 
			\frac{(n + \lambda)_{|p'|}}%
				{p'! V\big(2\xi, -\frac{|p'|+n+\lambda-1}{2}	\big)} \times \\
			&\times 
				\int_{\R^{n-1}_+ \times \R}
					\frac{\big(h'\big)^{p'}}%
						{(1 + \|h'\|_1)^{|p'|+n+\lambda}}
					\frac{e^{-2\xi \arccot(u)}}%
						{(1 + u^2)^{\frac{|p'| + n + \lambda + 1}{2}}}
							f(h',u) \dif h' \dif u,
	\end{align*}
	for every $p' \in \Z^{n-1}_+$ and $\xi \in \R$.
\end{theorem}
\begin{proof}
	We follow the notation and results from subsection~\ref{subsec:quasi-hyperbolic-Bergman} which, in this case, yields the operator $R_\lambda = W_\lambda^* \circ U_2 \circ U_1 \circ U_0$. Furthermore, Theorem~\ref{thm:Segal-Bargmann-QH} implies that we have $R_\lambda^* R_\lambda = B_{D_n,\lambda}$, the Bergman projection onto $\Aa^2_\lambda(D_n)$, and $R_\lambda R_\lambda^*$ is the identity operator acting on $\ell^2\big(\Z^{n-1}_+,L^2(\R)\big)$. Hence, we conclude that
	\begin{align*}
		R_\lambda \circ T^{(\lambda)}_a \circ R_\lambda^*
			&= R_\lambda B_{\B^n,\lambda} 
				M_a B_{\B^n,\lambda} R_\lambda^* \\
			&= R_\lambda R_\lambda^* R_\lambda 
				M_a R_\lambda^* R_\lambda R_\lambda^* \\
			&= R_\lambda M_a R_\lambda^* \\
			&= W_\lambda^* U_2 U_1 U_0 M_a U_0^* U_1^* U_2^* W_\lambda \\
			&= W_\lambda^* U_2 U_1 M_{a \circ \kappa} U_1^* U_2^* W_\lambda,
	\end{align*}
	where the last identity follows from the definition of $U_0$. On the other hand, $\kappa$ is the inverse of $\tau$ whose last $n$ components are given by $H$, and this together with \eqref{eq:widetildeH0} and \eqref{eq:widetildeHn-formula} yield
	\[
		a \circ \kappa(t',s,h) = f\bigg(h', 
				\frac{h_n}{1 + \|h'\|_1}\bigg),
	\]
	for every $t' \in \T^{n-1}$, $s \in \R_+$ and $h \in \R^{n-1}_+ \times \R$. Since this last function depends only on the variable $h$, the definition of $U_1$ implies that $M_{a\circ \kappa} = M_f$ commutes with $U_1$. Next, we observe that the operator $U_2$ involves a weighted change of variable and a straightforward computations shows that
	\[
		U_2 M_{a \circ \kappa} U_2^* = M_f.
	\]
	And so, we arrive to the identity
	\[
		R_\lambda \circ T^{(\lambda)}_a \circ R_\lambda^*
			= W_\lambda^* \circ M_f \circ W_\lambda.
	\]
	If we apply the formulas from Theorem~\ref{thm:Segal-Bargmann-QH}, then we obtain for every function $\psi \in \ell^2\big(\Z^{n-1}_+,L^2(\R)\big)$ that
	\begin{multline*}
		(W_\lambda^* \circ M_f \circ W_\lambda \psi)(p',\xi)
			= \frac{(n + \lambda)_{|p'|}}%
				{p'! V\big(2\xi, -\frac{|p'|+n+\lambda-1}{2}	
					\big)} \times \\
			\times 
				\Bigg(
					\int_{\R^{n-1}_+ \times \R}
						\frac{\big(h'\big)^{p'} f(h',u)}%
						{(1 + \|h'\|_1)^{|p'|+n+\lambda}}
						\frac{e^{-2\xi \arccot(u)}}%
						{(1 + u^2)^{\frac{|p'| + n + \lambda + 1}{2}}}
						 	\dif h' \dif u
				\Bigg) \psi(p',\xi),
	\end{multline*}
	for every $p' \in \Z^{n-1}_+$ and $\xi \in \R$. This finishes the proof.
\end{proof}

\begin{remark}\label{rmk:quasi-hyperbolic-spectral}
	Similar to the quasi-elliptic case, the function $\gamma_{a,\lambda}$ from Theorem~\ref{thm:quasi-hyperbolic-spectral} can be rewritten as
	\begin{multline}\label{eq:quasi-hyperbolic-spectral}
		\gamma_{a,\lambda}(p',\xi) = \\
			= \frac{\displaystyle \int_{\R^{n-1}_+ \times \R}
					\frac{\big(h'\big)^{p'}}%
					{(1 + \|h'\|_1)^{|p'|+n+\lambda}}
					\frac{e^{-2\xi \arccot(u)}}%
					{(1 + u^2)^{\frac{|p'| + n + \lambda + 1}{2}}}
					f(h',u) \dif h' \dif u}%
				{\displaystyle \int_{\R^{n-1}_+ \times \R}
					\frac{\big(h'\big)^{p'}}%
					{(1 + \|h'\|_1)^{|p'|+n+\lambda}}
					\frac{e^{-2\xi \arccot(u)}}%
					{(1 + u^2)^{\frac{|p'| + n + \lambda + 1}{2}}}
					\dif h' \dif u}, 
	\end{multline}
	for every $p' \in \Z^{n-1}_+$ and $\xi \in \R$. Hence, we obtain again a sort of normalized weighted special integral where the weight $f$ comes from the symbol $a$. It is worthwhile to compare the formula of this remark with that from Remark~\ref{rmk:quasi-elliptic-spectral}. It is readily observed that equation~\eqref{eq:quasi-hyperbolic-spectral} above extends equation~\eqref{eq:quasi-elliptic-spectral} of the quasi-elliptic case, in dimension $n-1$, by adding a further integral that involves the weight used to define the Romanovski-Routh polynomials.
	
	The last observation allows us to obtain spectral integral formulas for other symbols, among the quasi-hyperbolic ones, that have additional conditions.
\end{remark}

We note that the computations and results obtained so far hold for any dimension $n$. For the quasi-elliptic case, if we assume that $n = 1$, there is no essential change besides corresponding formulas simplified to integration over a single variable. However, for the quasi-hyperbolic case, the assumption $n=1$ does bring an important difference: there is no longer a toral action. More precisely, let us recall that the quasi-hyperbolic action on $D_n$, for $n \geq 1$, is given by
\begin{align*}
	\T^{n-1} \times \R_+ \times D_n &\longrightarrow D_n \\
		(t',s)\cdot z &= (s^\frac{1}{2} t'z', sz_n),
\end{align*}
which reduces to the so-called hyperbolic action for $n=1$ given by
\begin{align*}
	\R_+ \times \HH &\longrightarrow \HH  \\
		s\cdot z &= sz,
\end{align*}
where $\HH = D_1$ is the complex upper half-plane. With this notation at hand, a hyperbolic symbol on $\HH$ is, by definition, a function $a \in L^\infty(\HH)$ which is $\R_+$-invariant. This is equivalent to the existence of a measurable function $f : \R \rightarrow \C$ for which $a$ satisfies
\begin{equation}\label{eq:hyperbolic-symbol}
	a(z) = f\bigg(\frac{\re(z)}{\im(z)}\bigg),
\end{equation}
for almost every $z \in \HH$. The next result is a consequence of Theorem~\ref{thm:quasi-hyperbolic-spectral}, obtained by taking $n=1$, and provides alternative expressions for the known spectral integral formulas for the Toeplitz operators with hyperbolic symbols (see \cite{GKVHyperbolic,GQVJFA}).

\begin{corollary}\label{cor:hyperbolic-n=1}
	Let $a \in L^\infty(\HH)$ be a hyperbolic symbol and let $f : \R \rightarrow \C$ be the measurable function for which \eqref{eq:hyperbolic-symbol} holds. Then, for every $\lambda > -1$, the Toeplitz operator $T^{(\lambda)}_a$ acting on $\Aa^2_\lambda(\HH)$ satisfies
	\[
		R_\lambda \circ T^{(\lambda)}_a \circ R_\lambda^*
			= M_{\gamma_{a,\lambda}},
	\]
	where $R_\lambda$ is given by Theorem~\ref{thm:Segal-Bargmann-QH} and $M_{\gamma_{a,\lambda}}$ is the multiplier operator acting on $L^2(\R)$ with function $\gamma_{a,\lambda} : \R \rightarrow \C$ given by
	\begin{align*}
		\gamma_{a,\lambda}(\xi) 
			&= \frac{1}%
					{V\big(2\xi, -\frac{\lambda}{2}\big)} 
				\int_\R
					\frac{e^{-2\xi \arccot(u)}}%
						{(1 + u^2)^{\frac{\lambda}{2}+1}}
							f(u) \dif u  \\
			&= \frac{\displaystyle\int_\R
					\frac{e^{-2\xi \arccot(u)}}%
							{(1 + u^2)^{\frac{\lambda}{2}+1}}
								f(u) \dif u}%
					{\displaystyle\int_\R
						\frac{e^{-2\xi \arccot(u)}}%
						{(1 + u^2)^{\frac{\lambda}{2}+1}}\dif u} 
	\end{align*}
	for every $\xi \in \R$.
\end{corollary}

If we now consider in Theorem~\ref{thm:quasi-hyperbolic-spectral}  symbols that depend on $\widetilde{H}_0$ only, as defined in \eqref{eq:widetildeH0}, then we arrive at the next result. Again, it is interesting to look at Remark~\ref{rmk:quasi-hyperbolic-spectral}. We will denote by $\Ta^{(\lambda)}(L^\infty(D)^G)$ the $C^*$-algebra generated by Toeplitz operators with symbols in $L^\infty(D)^G$ ($G$-invariant ones), where $(G,D)$ denotes either of the pairs $(\T^n, \B^n)$ or $(\T^{n-1} \times \R_+, D_n)$. Note that the last claim in the next result follows from a comparison with Theorem~\ref{thm:quasi-elliptic-spectral} and Remark~\ref{rmk:quasi-elliptic-spectral}.

\begin{corollary}\label{cor:quasi-hyperbolic-H_0}
	Let $a \in L^\infty(D_n)$ be a quasi-hyperbolic moment map symbol that satisfies $a = f \circ \widetilde{H}_0$ where $\widetilde{H}_0$ is given by \eqref{eq:widetildeH0} and $f : \R_+^{n-1} \rightarrow \C$ is an essentially bounded measurable function. Then, for every $\lambda > -1$, the Toeplitz operator $T^{(\lambda)}_a$ acting on $\Aa^2_\lambda(D_n)$ satisfies
	\[
		R_\lambda \circ T^{(\lambda)}_a \circ R_\lambda^*
			= M_{\gamma_{a,\lambda}},
	\]
	where $R_\lambda$ is given by Theorem~\ref{thm:Segal-Bargmann-QH} and $M_{\gamma_{a,\lambda}}$ is the multiplier operator acting on $\ell^2\big(\Z^{n-1}_+,L^2(\R)\big)$ with function $\gamma_{a,\lambda} : \Z_+^{n-1} \times \R \rightarrow \C$ given by
	\[
		\gamma_{a,\lambda}(p',\xi) 
		= \frac{\displaystyle \int_{\R^{n-1}_+}
			\frac{\big(h'\big)^{p'}}%
			{(1 + \|h'\|_1)^{|p'|+n+\lambda}}
				f(h') \dif h'}%
		{\displaystyle \int_{\R^{n-1}_+}
			\frac{\big(h'\big)^{p'}}%
			{(1 + \|h'\|_1)^{|p'|+n+\lambda}}
				\dif h'}, 
	\]
	for every $p' \in \Z_+^{n-1}$ and $\xi \in \R$. In particular, the function $\gamma_{a,\lambda}$ is independent of the variable $\xi \in \R$. Furthermore, it follows that, for every $n \geq 2$, the $C^*$-algebra $\Ta^{(\lambda)}(L^\infty(D_n)^{\T^{n-1} \times \R_+})$ contains a $C^*$-subalgebra isomorphic to $\Ta^{(\lambda)}(L^\infty(\B^{n-1})^{\T^{n-1}})$.
\end{corollary}

\subsection*{Acknowledgments}
This research was partially supported by Conahcyt Grants 280732 and 61517.

\end{document}